\documentclass[letterpaper,10pt]{amsart}
\usepackage{indentfirst} 
\usepackage{amssymb}
\usepackage{mathtools} 
\usepackage{mathabx} 
\usepackage{amsthm}  
\usepackage{thmtools} 
\usepackage{enumitem} 
\usepackage[colorlinks=true]{hyperref} 
\usepackage[usenames,dvipsnames]{xcolor} 

\makeatletter

\def\MRbibitem{\@ifnextchar[\my@lbibitem\my@bibitem}

\def\mybiblabel#1#2{\@biblabel{{\hyperref{http://www.ams.org/mathscinet-getitem?mr=#1}{}{}{#2}}}}

\def\myhyperanchor#1{\Hy@raisedlink{\hyper@anchorstart{cite.#1}\hyper@anchorend}}

\def\my@lbibitem[#1]#2#3#4\par{%
    \item[\mybiblabel{#2}{#1}\myhyperanchor{#3}\hfill]#4%
    \@ifundefined{ifbackrefparscan}{}{\BR@backref{#3}}%
    \if@filesw{\let\protect\noexpand\immediate
       \write\@auxout{\string\bibcite{#3}{#1}}}\fi\ignorespaces%
}

\def\my@bibitem#1#2#3\par{%
    \refstepcounter\@listctr
    \item[\mybiblabel{#1}{\the\value\@listctr}\myhyperanchor{#2}\hfill]#3%
    \@ifundefined{ifbackrefparscan}{}{\BR@backref{#2}}%
    \if@filesw\immediate\write\@auxout
        {\string\bibcite{#2}{\the\value\@listctr}}\fi\ignorespaces%
}

\makeatother

\declaretheorem{theorem}
\declaretheorem[sibling=theorem]{lemma}

\declaretheorem[sibling=theorem]{proposition}
\declaretheorem[sibling=theorem,style=remark]{remark}

\newcommand{\R}{\mathbb{R}}

\renewcommand{\P}{\mathbb{P}}

\newcommand{\cM}{\mathcal{M}}
\newcommand{\cP}{\mathcal{P}}

\DeclareMathOperator{\per}{per}

\setlist[enumerate,1]{label={\upshape(\alph*)},ref=\alph*}
\setlist[enumerate,2]{label={\upshape(\arabic*)},ref=\arabic*}

\begin{document}

\title{An ergodic theorem for permanents of oblong matrices}
\date{November, 2014 (first version); October, 2016 (final version).}


\begin{thanks}
{J.B., G.I.\ and M.P.\ were partially supported by the Center of Dynamical Systems and Related Fields c\'odigo ACT1103 and FONDECYT projects 1140202, 1110040, and 1140988, respectively.}
\end{thanks}

\author[J.~Bochi]{Jairo Bochi}
\author[G.~Iommi]{Godofredo Iommi}
\author[M.~Ponce]{Mario Ponce}

\begin{abstract}
We form a sequence of oblong matrices by evaluating an integrable vector-valued function along the orbit of an ergodic dynamical system. We obtain an almost sure asymptotic result for the permanents of those matrices. We also give an application to symmetric means.
 \end{abstract}

\maketitle

\fbox{
	\begin{minipage}{0.9\textwidth}
		\textbf{Disclaimer:} After the completion of this paper, we were informed by the referee that our Theorem~\ref{t.main} is a particular case of \cite[Theorem~U]{Aetc}, which was subsequently extended in the paper \cite{DG}.
	\end{minipage}
}

\bigskip

\section{An ergodic theorem for permanents}
If $A=(a_{i,j})$ is a $m \times n$ real matrix with $m \le n$, then its \emph{permanent} is 
defined by
$$
\per A \coloneqq  \sum_{\tau} a_{1, \tau(1)} \cdots a_{m,\tau(m)} \, ,
$$
where the sum is taken over all one-to-one functions $\tau \colon \{1, \dots, m\} \to  \{1, \dots, n\}$.
The number of such functions is the ``falling power''
$$
n^{\downarrow m} \coloneqq n (n-1) \cdots (n-m+1) \, .
$$
We are interested in the average value $(\per A)/n^{\downarrow m}$ of the products $a_{1, \tau(1)} \cdots a_{m,\tau(m)}$. Our main result is:

\begin{theorem}\label{t.main}
Let $(\Omega,\P)$ be a probability space, and let $T \colon \Omega \to \Omega$ be an ergodic measure-preserving transformation.
Let $f_1$, \dots, $f_m \in L^1(\P)$.
For each $n \ge m$ and $\omega \in \Omega$, consider the matrix
\begin{equation}\label{e.oblong}
A(n,\omega) \coloneqq 
\begin{pmatrix}
f_1(\omega) &f_1 (T\omega) &\cdots\cdots &f_1(T^{n-1} \omega) \\
\vdots  &   &                 &\vdots  \\
f_m(\omega) &f_m (T\omega) &\cdots\cdots &f_m(T^{n-1} \omega)
\end{pmatrix}
\end{equation}
Then for $\P$-almost every $\omega$ we have
\begin{equation}\label{e.our_formula}
\lim_{n \to \infty}
\frac{1}{n^{\downarrow m}} \per A(n,\omega)
=
\prod_{i=1}^m \int f_i \, \mathrm{d}\P \, .
\end{equation}
\end{theorem}

In other words, consider a random infinite oblong matrix
whose columns are given by an ergodic stationary process $X_1$, $X_2$, \dots
taking values in $\R^m$. Then the permanent of the truncated $m \times n$ matrix
is asymptotically equal to ${n^{\downarrow m}}\lambda $, where $\lambda$ is the product of the expectations of the entries of $X_1$. 

\smallskip

Asymptotic results for the permanents of similar sequences of matrices have been obtained under much stronger assumptions on the distribution of the entries. See for example the work of  Rempa{\l}a and Weso{\l}owski \cite{RW} where several results, like the Central Limit Theorem, are obtained under strong distribution assumptions. Along the same lines we can mention the work of  Borovskikh and  Korolyuk  \cite{BK} and that of  Kaneva and Korolyuk~\cite{KK}.

\smallskip

Technically, one of the main novelties  of our approach 
is that we derive asymptotic information about permanents
from the following Binet--Minc formula \cite[Theorem~1.2, p.~120]{Minc}:

\begin{proposition}\label{p.BM}
Let $A = (a_{i,j})$ be a $m \times n$ matrix with $m \le n$.
Let $[m] \coloneqq \{1, 2, \dots, m\}$. 
For each nonempty subset $I \subset [m]$, write
\begin{equation}\label{e.sum}
s_I \coloneqq \sum_{j = 1}^n \prod_{i\in I} a_{i,j} \, .
\end{equation}
Let $\cP_m$ denote the set of all partitions of $[m]$ into nonempty subsets.
Then
\begin{equation}\label{e.binet-minc}
\per A = 
\sum_{P \in \cP_m} (-1)^{m - |P|} \prod_{I \in P} \big(|I| - 1\big)! \, s_I \, .
\end{equation}
\end{proposition}

As an example, in the case $m=3$, formula \eqref{e.binet-minc} becomes
$$
\per A = 
s_{\{1\}} s_{\{2\}} s_{\{3\}} - s_{\{1\}} s_{\{2,3\}} - s_{\{2\}} s_{\{1,3\}} -  s_{\{3\}} s_{\{1,2\}} + 2 s_{\{1,2,3\}} \, ,
$$
a relation obtained by Binet in 1812.

The Binet--Minc formula is relatively unpopular because there is a 
another formula, due to Ryser,
which is more efficient for computational purposes: see \cite[\S~7.2--7.3]{Minc}.

\begin{remark}
We found out that in 1968, thus 10 years before Minc, 
Crapo \cite{Crapo} used the (then recent) powerful combinatorial theory of M\"obius inversion
to give extremely short proofs of both Ryser's formula and Proposition~\ref{p.BM}.
Crapo attributes Proposition~\ref{p.BM} to J.E.~Graver and W.~Gustin (unpublished).
In spite of all this, we will keep calling it Binet--Minc formula.
\end{remark}


\section{Proof of the theorem}

The idea of the proof of Theorem~\ref{t.main} is 
to show that when $A=A(n,\omega)$ all terms in the sum in \eqref{e.binet-minc}
are negligible compared to $n^{\downarrow m}$, except for the term
that comes from the partition $P = \{\{1\},\{2\},\dots,\{m\}\}$, 
whose behavior is described by Birkhoff's ergodic theorem.
In order to prove that the other terms are indeed negligible, we will use the following result:

\begin{lemma}[Aaronson]\label{l.Aaronson}
If  $0 < p < 1$ and $f \in L^p(\P)$ then for $\P$-a.e.\ $\omega$,
$$
\lim_{n \to \infty} \frac{1}{n^{1/p}} \sum_{j=0}^{n-1} f(T^j \omega) = 0 \, .
$$
\end{lemma}

The lemma is a particular case of \cite[Prop.~2.3.1, p.~65]{Aa_book}
corresponding to the function $a(x) \coloneqq x^p$.
(See also \cite{Aa_paper}.)

\begin{proof}[Proof of the  Theorem~\ref{t.main}]
Let $(\Omega,\P)$, $T$, and $f_1$, \dots, $f_m$ be as in the statement of the theorem.
For each integer $n \ge m$, each $\omega \in \Omega$,
and each nonempty subset $I$ of $[m] \coloneqq \{1,\dots,m\}$, 
let $s_I(n,\omega)$ be the sum \eqref{e.sum} 
corresponding to the matrix $A = A(n,\omega)$ defined by \eqref{e.oblong}.
We can write this expression as a Birkhoff sum
$$
s_I(n,\omega) = \sum_{j=0}^{n-1} f_I(T^j \omega)
$$
of the function $f_I \coloneqq \prod_{i \in I} f_i$.
By the H\"older inequality,
$$
\left(\int |f_I|^{1/|I|} \, \mathrm{d}\P \right)^{|I|} \le \prod_{i \in I}  \int |f_i| \, \mathrm{d}\P \, ,
$$
and in particular $f_I$ belongs to the space $L^{1/|I|}(\P)$.
So for a.e.~$\omega$,
it follows from  Birkhoff's ergodic theorem and Lemma~\ref{l.Aaronson} that
$$
\lim_{n \to \infty} \frac{s_I(n,\omega)}{n^{|I|}}  = 
\begin{cases}
	\int f_I \, \mathrm{d}\P &\quad \text{if } |I| = 1 \, , \\ 
	0 &\quad \text{if } |I| > 1 \, .
\end{cases}
$$
Therefore, by Binet--Minc formula \eqref{e.binet-minc},
$$
\frac{\per A(n,\omega)}{n^m} = 
\sum_{P \in \cP_m} (-1)^{m - |P|} \prod_{I \in P} \big(|I| - 1\big)! \, \frac{s_I(n,\omega)}{n^{|I|}} 
\to \prod_{i=1}^m 	\int f_I \, \mathrm{d}\P 
$$
as $n \to \infty$,
because the only $P$ that contributes to the limit is the partition into $m$ singletons.
Since $n^{\downarrow m} / n^m \to 1$ as $n \to \infty$,
relation \eqref{e.our_formula} follows.
\end{proof}

\begin{remark} 
The aforementioned works \cite{BK,KK,RW} consider sequences of matrices
where the number of rows $m$ is allowed to depend on the number of columns $n$.
This more general situation cannot be handled with our technique, at least without 
further assumptions.
\end{remark}

\section{An Ergodic Theorem for Symmetric Means}

Let $1 \le m \le n$ be integers.
Recall that the \emph{elementary symmetric polynomial of degree $m$ in $n$ variables} is defined as 
\[
E_m(x_1, x_2, \dots , x_n) \coloneqq\sum_{1 \leq i_1 < i_2 < \dots <i_m \leq n} x_{i_{1}}x_{i_{2}} \cdots x_{i_{m}} \, .
\]
In the case that $x_1$, \dots, $x_n$ are nonnegative real numbers,
we define their
\emph{$m$-th symmetric mean} by
\begin{equation*}
\cM_m(x_1, x_2, \dots , x_n) \coloneqq \left( \frac{E_m(x_1,\dots,x_n)}{\left( {n \atop m } \right)}   \right)^{1/m}.
\end{equation*}
Symmetric means generalize both the arithmetic mean, given by $\cM_1(x_1,\dots,x_n)$, and the geometric mean, given by $\cM_n(x_1,\dots,x_n)$.  
A classical result of Maclaurin says that $\cM_m(x_1,\dots,x_n)$ is nonincreasing with respect to $m$: see \cite[\S~2.22]{HLP}.

\smallskip

It turns out that Theorem \ref{t.main} can be used to describe the asymptotic behavior of symmetric means. The result is as follows:

\begin{proposition}\label{p.coro}
Let $(\Omega,\P)$ be a probability space, and let $T \colon \Omega \to \Omega$ be an ergodic measure-preserving transformation. Fix a measurable function $f \ge 0$
and an integer $m \ge 0$.
\begin{enumerate}
\item\label{i.a} 
If $f \in L^1(\P)$ and $m > 0$ then for $\P$-almost every $\omega \in \Omega$ we have
\begin{equation*}
\lim_{n \to \infty} 
\cM_m \big( f(\omega), f(T\omega), \dots, f(T^{n-1}\omega)\big)
=  \int  f\,\mathrm{d}\P.
\end{equation*}
\item\label{i.b}
If  $ \log f \in L^1(\P)$ then for $\P$-almost every $\omega \in \Omega$ we have
\begin{equation*}
\lim_{n \to \infty} 
\cM_{n-m} \big( f(\omega), f(T\omega), \dots, f(T^{n-1}\omega)\big)
= \exp \left( \int \log f \,  \mathrm{d}\P \right).
\end{equation*}
\end{enumerate}
\end{proposition}

\begin{remark}
A much more general result 
due to Hal\'{a}sz and Sz\'{e}kely \cite{HS}
describes the what happens when $m = m(n)$ is such that $m(n)/n$ converges to some $c \in [0,1]$.
(Notice that the independence assumption in that paper is not actually used and ergodicity suffices.)
\end{remark}

The first part of the proposition is just a corollary of Theorem~\ref{t.main}:

\begin{proof}[Proof of part~(\ref{i.a})]
Consider functions $f_1$, \dots, $f_m$ all equal to $f$;
then the the matrix defined by \eqref{e.oblong} has permanent
$$
\per A(n,\omega) = 
m! \, E_m \big( f(\omega), f(T\omega), \dots, f(T^{n-1}\omega)\big) \, .
$$
Therefore Theorem~\ref{t.main} allows us to conclude.
\end{proof}

\begin{remark}
Binet--Minc formula \eqref{e.binet-minc} applied to
matrices with equal rows yields a way of 
expressing elementary symmetric polynomials in terms of power sums.
Such expressions are equivalent to Newton's identities (see \cite[p.~95--96, 251]{Merris}).
So it should be possible to deduce part~(\ref{i.a}) of Proposition~\ref{p.coro}
directly from Newton's identities and Aaronson's Lemma~\ref{l.Aaronson}.
\end{remark}

To prove the second part of Proposition~\ref{p.coro}, we need the following fact:
\begin{lemma}\label{l.max}
If $f \in L^1(\P)$ then for $\P$-a.e.\ $\omega$,
$$
\lim_{n \to \infty} \frac{1}{n} \ \max_{j = 0,\dots,n-1}  \big| f(T^j \omega) \big|
= 0\, .
$$
\end{lemma}

\begin{proof}
By Birkhoff's ergodic theorem, $f(T^n\omega)/n \to 0$,
and the lemma is a straightforward consequence.
\end{proof}

\begin{proof}[Proof of part~(\ref{i.b})]
The case $m=0$ is rather simple; indeed,
$$
\log \cM_{n}\big( f(\omega), \dots, f(T^{n-1}\omega)\big) =  \cM_1\big( \log f(\omega), \dots, \log f(T^{n-1}\omega)\big),
$$
so the result follows from Birkhoff's ergodic theorem.
We reduce the case $m>0$ to the above by using the relation
$$
\frac{\cM_{n-m}\big( f(\omega), \dots, f(T^{n-1}\omega)\big)}{\cM_{n}\big( f(\omega), \dots, f(T^{n-1}\omega)\big)} = 
\left[ \cM_{m}\big( 1/f(\omega), \dots, 1/f(T^{n-1}\omega)\big) \right]^{\frac{m}{n-m}}
\eqqcolon (\star)
\, .
$$
Since the $m$-th symmetric mean of a list of nonnegative numbers is between their minimum and their maximum, we have
$$
\left| \log (\star)  \right| \le 
\frac{m}{n-m} \ \max_{j = 0,\dots,n-1}  \big| \log f(T^j \omega)\big|  \, ,
$$
which by Lemma~\ref{l.max} converges almost everywhere to $0$ as $n \to \infty$.
So the result follows.
\end{proof}


\bigskip

\bigskip

\begin{small}
	\noindent
	\begin{tabular}{lll}
		\textsc{Jairo Bochi} & 
		\textsc{Godofredo Iommi} & 
		\textsc{Mario Ponce} \\ 
		\email{\href{mailto:jairo.bochi@mat.puc.cl}{jairo.bochi@mat.puc.cl}} &
		\email{\href{mailto:giommi@mat.puc.cl}{giommi@mat.puc.cl}} &		
		\email{\href{mailto:mponcea@mat.puc.cl}{mponcea@mat.puc.cl}} \\
		\href{http://www.mat.uc.cl/~jairo.bochi}{www.mat.uc.cl/$\sim$jairo.bochi} & 
		\href{http://www.mat.uc.cl/~giommi}{www.mat.uc.cl/$\sim$giommi} & 		
		\href{http://www.mat.uc.cl/~mponcea}{www.mat.uc.cl/$\sim$mponcea} 	
	\end{tabular}
	
	\bigskip
	
	\noindent
	\textsc{Facultad de Matem\'aticas, Pontificia Universidad Cat\'olica de Chile}
	
	\noindent
	\textsc{Avenida Vicu\~na Mackenna 4860, Santiago, Chile}
\end{small}

\end{document}